\documentclass[11pt]{amsart}

\usepackage{amscd,amssymb,amsopn,amsmath,amsthm,mathrsfs,graphics,amsfonts,enumerate,verbatim,calc
}
\usepackage{bbm}
\usepackage[all,cmtip]{xy}
\usepackage[all]{xy}
\usepackage{tikz}
\usepackage{lscape}
\usetikzlibrary{matrix,arrows,decorations.pathmorphing,cd}

\usepackage{hyperref}
\hypersetup{colorlinks=true}
\usepackage{cleveref}

\usepackage{scalerel}
\usepackage{stackengine,wasysym}
\usetikzlibrary {positioning}
\usetikzlibrary{patterns}
\usetikzlibrary{calc}
\definecolor {processblue}{cmyk}{0.96,0,0,0}

\usepackage{subfigure}

\usepackage{comment} 

\usepackage{ dsfont }

\usepackage{color}

\usepackage[OT2,OT1]{fontenc}
\newcommand\cyr{%
\renewcommand\rmdefault{wncyr}%
\renewcommand\sfdefault{wncyss}%
\renewcommand\encodingdefault{OT2}%
\normalfont
\selectfont}
\DeclareTextFontCommand{\textcyr}{\cyr}

\usepackage{amssymb,amsmath}

\DeclareFontFamily{OT1}{rsfs}{}
\DeclareFontShape{OT1}{rsfs}{n}{it}{<-> rsfs10}{}
\DeclareMathAlphabet{\mathscr}{OT1}{rsfs}{n}{it}

\numberwithin{equation}{section}
\newtheorem{theorem}{Theorem}[section]
\newtheorem{lemma}[theorem]{Lemma}
\newtheorem{cor}[theorem]{Corollary}

\newtheorem{prop}[theorem]{Proposition}

\theoremstyle{definition}
\newtheorem{defn}[theorem]{Definition}
\theoremstyle{remark}
\newtheorem{remark}[theorem]{Remark}

\renewcommand{\tilde}{\widetilde}

\newcommand{\lk}{\operatorname{lk}}

\newcommand{\Tor}{\operatorname{Tor}}

\newcommand{\depth}{\operatorname{depth}}

\begin{document}
\title[Minimal Cohen-Macaulay Complexes]{Minimal Cohen-Macaulay Simplicial Complexes}

\author[Dao]{Hailong Dao}
\email[Hailong Dao]{hdao@ku.edu}
\author[Doolittle]{Joseph Doolittle}
\email[Joseph Doolittle]{jdoolitt@ku.edu}
\author[Lyle]{Justin Lyle}
\email[Justin Lyle]{justin.lyle@ku.edu}
\address{Department of Mathematics\\
University of Kansas\\
Lawrence, KS 66045-7523 USA}
\date{\today}



\begin{abstract}
We define and study the notion of a minimal Cohen-Macaulay simplicial complex. We prove that any Cohen-Macaulay complex is shelled over a minimal one in our sense, and we give sufficient conditions for a complex to be minimal Cohen-Macaulay. We show that many interesting examples of Cohen-Macaulay complexes in combinatorics are minimal, including Rudin's ball, Ziegler's ball, the dunce hat, and the non-partitionable Cohen-Macaulay complex of \cite{DG16}. We further provide various ways to construct such complexes.

\end{abstract}

\maketitle

\section{Introduction}

In this work, we introduce and study the notion of a minimal Cohen-Macaulay complex. Fix a field $k$. Let $\Delta$ be a simplicial complex. We say $\Delta$ is minimal Cohen-Macaulay (over $k$) if it is Cohen-Macaulay and removing any facet from the facet list of $\Delta$ results in a complex which is not Cohen-Macaulay. See Section \ref{sec2} for precise definitions.

For the rest of the paper we shall write CM for Cohen-Macaulay. We first observe a crucial fact. 

\begin{theorem}\label{t1}
Any CM complex is shelled over a minimal CM complex. $($Theorem \ref{thm1}$)$ 
\end{theorem}

Thus, in a strong sense, understanding CM complexes amounts to understanding the minimal ones. We support this claim by demonstrating that many interesting examples of CM complexes in combinatorics are minimal. Theorem \ref{t1} also puts shellable complexes in a broader context: they are precisely complexes shelled over the empty one. Its proof relies on a simple but somewhat surprising statement (Lemma \ref{isshelled}), which might be of independent interest.

Below is a collection of our main technical results which establish various necessary and sufficient conditions for a complex to be minimal CM.

\begin{theorem} The following statements hold.
\begin{enumerate}
\item[$(1)$] A minimal CM complex is acyclic. $($Corollary \ref{cor2}$)$
\item[$(2)$] Let $\Delta$ be CM and $i$-fold acyclic. If no facet of $\Delta$ contains more than $i-1$ boundary ridges, then $\Delta$ is minimal CM. $($Theorem \ref{thm3}$)$
\item[$(3)$] If $\Delta$ is a ball, then $\Delta$ is minimal CM if and only if it is strongly non-shellable in the sense of \cite{Zi98}. $($Proposition \ref{strongnonshell}$)$
\item[$(4)$] If $\Delta$ is minimal CM and $\Gamma$ is CM, then $\Delta \star \Gamma$ is minimal CM. $($Theorem \ref{thm4}$)$

\end{enumerate}
\end{theorem}

In Section \ref{sec2}, we give the formal definitions, provide background, and set notation. Section \ref{main} contains the proofs of Theorem \ref{thm1}, Theorem \ref{thm3}, Corollary \ref{cor2} and Proposition \ref{strongnonshell}. In Section \ref{newmCM}, we provide many ways to build new minimal Cohen Macaulay from old ones, such as gluing (Corollary \ref{glue1}, Proposition \ref{glue2}) and taking joins (Theorem \ref{thm4}).  In the last section, we use our results to examine many classical and recent examples of Cohen-Macaulay complexes from the literature and show that they are minimal.


\section{Background and Notation}\label{sec2}

Once and for all, fix the base field $k$.  We let $\tilde{H}_i$ denote $i$th simplicial or singular homology, as appropriate, always with coefficients in $k$. We use $\tilde{\chi}$ for reduced Euler characteristic. Throughout this paper, we let $\Delta$ be a simplicial complex of dimension $d-1$ with facet list $\{F_1,\dots,F_e\}$, and we denote by $\Delta_{F_i}$ the subcomplex of $\Delta$ with facet list $\{F_1,\dots,F_{i-1},F_{i+1},\dots,F_e\}$. We write $f_i(\Delta)$ for the number of $i$-dimensional faces of $\Delta$, and $h_i(\Delta)$ for the $i$th entry of the $h$-vector of $\Delta$; so $h_i(\Delta)=\sum^i_{k=0} \binom{d-k}{i-k}(-1)^{i-k}f_{k-1}(\Delta)$.  In particular, we note that $h_d(\Delta)=\sum^d_{k=0}(-1)^{d-k}f_{k-1}(\Delta)=(-1)^{d-1}\tilde{\chi}(\Delta)$.  The $\depth$ of $\Delta$ is, by definition, the depth of the Stanley-Reisner ring $k[\Delta]$ of $\Delta$.  We say $\Delta$ is CM if $\depth \Delta=d$.  The following consequence of Hochster's formula (\cite[Theorem 5.3.8]{BH98}) is an extension of Reisner's famous criterion for Cohen-Macaulayness (\cite[Theorem]{Re76}) and gives a combinatorial characterization of $\depth$. 

\begin{prop}\label{depth}

$\depth \Delta \ge \ell$ if and only if $\tilde{H}_{i-1}(\lk_{\Delta}(T))=0$ whenever $i+|T|<\ell$.

\end{prop}

We use $\Delta^{(i)}:=\{ \sigma \in \Delta \colon |\sigma| \le i+1 \}$ to denote the $i$-skeleton of $\Delta$, and we note that $\depth \Delta=\max\{i \mid \Delta^{(i-1)} \mbox{ is CM}\}$.

The following definition gives the main focus of this paper.

\begin{defn}
We say $\Delta$ is \textbf{minimal CM} if $\Delta$ is CM but $\Delta_{F_i}$ is not CM for any $i$.
\end{defn}

The following related concept provides an extension of the notion of shellability.

\begin{defn}

We say $\Delta_F$ to $\Delta$ is a \textbf{shelling move} if $\langle F \rangle \cap \Delta_F$ is pure of dimension $|F|-2$.  If $\Gamma$ is a subcomplex of $\Delta$ generated by facets of $\Delta$, we say $\Delta$ is \textbf{shelled over $\Gamma$} if there exists a sequence of shelling moves taking $\Gamma$ to $\Delta$.

\end{defn}

We note that shellable complexes are exactly those which are shelled over $\varnothing$.

We will use the following definitions in the later sections. 


\begin{defn}
We say that $\Delta$ is $l$-fold acyclic if $\lk_{\Delta}(\sigma)$ is acyclic whenever $|\sigma|<l$.
\end{defn}

\begin{defn}
A \textit{pseudomanifold} is a simplicial complex that is pure, has a connected dual graph, and every ridge is in exactly two facets.
\end{defn}

\begin{defn}[{See \cite{MT09}}]

We say $\Delta$ satisfies Serre's condition $(S_{\ell})$ if $\ell=1$ or if $\ell \ge 2$ and $\tilde{H}_{i-1}(\lk_{\Delta}(T))=0$ whenever $i+|T|<d$ and $0 \le i<\ell$.

\end{defn}

Finally, we recall Alexander duality and some of its features.

\begin{defn}

The \textbf{Alexander dual} of $\Delta$ is the simplicial complex $\Delta^{\vee}:=\{F \subseteq [n] \mid [n]-F \notin \Delta\}$.
\end{defn}

We now list the connection between several combinatorial properties and the corresponding properties of the Alexander dual ideal. 

\begin{theorem}\label{dualstuff}
\
\begin{enumerate}

\item[$(1)$] $\Delta$ is CM if and only if $I_{\Delta^{\vee}}$ has a linear resolution $($\cite[Theorem 3]{ER98}$)$. 

\item[$(2)$] $\Delta$ is shellable and pure if and only if $I_{\Delta^{\vee}}$ has linear quotients $($\cite[Theorem 1.4 (c)]{HH04}$)$. 

\item[$(3)$] $\Delta_F$ to $\Delta$ is a shelling move if and only if $(I_{\Delta_F^{\vee}}:f)$ is generated in degree $1$, where $f=\prod_{i \in [n]-F} x_i$ is the monomial corresponding to $F$ in $I_{\Delta^{\vee}}$.

\end{enumerate}

\end{theorem}

\begin{remark}
We remark that this correspondence leads to a dual version of minimal CM: an ideal with linear resolution is minimal if removing any minimal generator results in one which does not have a linear resolution. While we won't pursue this point of view here, it was quite helpful to us; for instance, it suggested the original (algebraic) proof of Lemma \ref{isshelled}. 
\end{remark}

\section{Main Results}\label{main}

In this section we prove most of our main technical results regarding minimal CM complexes. We start with the following lemma.

\begin{lemma}\label{isshelled}

Suppose $\Delta$ satisfies Serre's condition $(S_2)$.  Then, for any facet $F \in \Delta$, $\Delta$ is shelled over $\Delta_F$.

\end{lemma}

For the convenience of the reader, we provide two proofs of this lemma, one combinatorial and one algebraic. We begin with the combinatorial proof.

\begin{proof}

The claim is clear if $\Delta$ is a simplex, so we suppose this is not the case.
Noting that $\Delta$ is pure (since it satisfies $(S_2)$), we want to show that \(\langle F \rangle \cap \Delta_{F}\) is pure of dimension \(d-2\).

Let \(\sigma\) be a facet of \(\langle F \rangle \cap \Delta_{F}\). Then \(F \setminus \sigma \) is a facet of \(\lk_\Delta(\sigma)\), but there is also a facet \(G \in \Delta_{F}\) such that \(G \setminus \sigma\) is a facet of \(\lk_\Delta(\sigma)\). If $v \in F \setminus \sigma \cap G \setminus \sigma$, then $\sigma \cup \{v\} \in \langle F \rangle \cap \Delta_{F}$, contradicting that $\sigma$ is a facet.  Thus $F \setminus \sigma \cap G \setminus \sigma=\varnothing$, which means \(\lk_\Delta(\sigma)\) is disconnected. Since \(\Delta\) satisfies \((S_2)\), this must mean that \(\sigma\) is \(d-2\)-dimensional. Ergo, \(\langle F \rangle \cap \Delta_{F}\) is pure of dimension $d-2$, and so \(\Delta\) is shelled over \(\Delta_{F}\).

\end{proof}


We now provide the algebaric proof.

\begin{proof}


Set $I=I_{\Delta_F}$ and let $f$ be the monomial (of degree $c:=n-d$) in $I_{\Delta^{\vee}}$ corresponding to the facet $F$.  Then $I_{\Delta^{\vee}}=(I,f)$ and we have a short exact sequence of graded $S$-modules

\[0 \to S/(I:f)(-c) \to S/I \to S/(I,f) \to 0.\]
This exact sequence induces the following long exact sequence in $\Tor$: 

\[\begin{tikzpicture}[descr/.style={fill=white,inner sep=1.5pt}]
        \matrix (m) [
            matrix of math nodes,
            row sep=1em,
            column sep=1.8em,
            text height=1.5ex, text depth=0.25ex
        ]
        { S/(I:f)(-c) \otimes_S k & S/I \otimes_S k & S/(I,f) \otimes_S k & 0 \\
           \Tor_1^{S}(S/(I:f)(-c),k) & \Tor_1^{S}(S/I,k) & \Tor_1^{S}(S/(I,f),k) & \mbox{} \\
          \Tor_2^{S}(S/(I:f)(-c),k)  & \Tor_2^{S}(S/I,k) & \Tor_2^{S}(S/(I,f),k) & \mbox{} \\
          \mbox{}  & \mbox{} & \hspace{2cm} & \mbox{} \\
        };

        \path[overlay,->, font=\scriptsize,>=latex]
        (m-1-1) edge (m-1-2)
        (m-1-2) edge (m-1-3)
        (m-1-3) edge (m-1-4);
        \path[overlay,->, font=\scriptsize,>=latex]
        (m-2-3) edge[out=365,in=185] (m-1-1)
        (m-2-2) edge (m-2-3)
        (m-2-1) edge (m-2-2)
        (m-3-3) edge[out=365,in=185] (m-2-1);
        \path[overlay,->, font=\scriptsize,>=latex]
        (m-3-2) edge (m-3-3)
        (m-3-1) edge (m-3-2);
        \path[overlay,->, font=\scriptsize,>=latex]
        (m-4-3) edge[out=365,in=185,dashed] (m-3-1);   
\end{tikzpicture}\]

As $S/I \otimes_S k \cong S/(I,f) \otimes_S k \cong k$, the map $S/I \otimes_S k \to S/(I,f) \otimes_S k$ is an isomorphism. So $\Tor^S_1(S/(I,f),k) \cong k^{\mu(I)+1}$ and $\Tor^S_1(S/(I,f),k) \cong k^{\mu(I)}$. By additivity of dimensions, it follows that the map $\Tor^S_1(S/I,k) \to \Tor^S_1(S/(I,f),k)$ is injective.  Hence the map $\Tor_2(S/(I,f),k) \to \Tor_1^S(S/(I:f)(-c),k)$ is surjective.  But since $\Delta$ satisfies $(S_2)$, $S/(I,f)$ has linear first syzygy (see \cite[Corollary 3.7]{Ya00}), so $\Tor_2^S(S/(I,f),k) \cong k^{\beta^S_1(I,f)}(-c-1)$. Hence it must be that $\Tor_1^S(S/(I:f)(-c),k)$ is generated in degree $-c-1$.  Thus $(I:f)$ is generated in degree $1$, and the claim follows from Theorem \ref{dualstuff} (3).

\end{proof}

\begin{theorem}\label{thm1}
If \(\Delta\) is a CM complex, then there is a minimal CM complex \(\Gamma\) so that \(\Delta\) is shelled over \(\Gamma\).
\end{theorem}

\begin{proof}
Since \(\Delta\) is CM, it also satisfies \((S_2)\). We can then apply Lemma \ref{isshelled} to conclude that, for every facet \(F\) of \(\Delta\), \(\Delta\) is shelled over \(\Delta_{F}\). If none of these is  CM, then $\Delta$ is minimal CM by definition. If not, we may continue this process to eventually reach a minimal one. 
\end{proof}

\begin{remark}
It is not hard to see that a given CM complex can be shelled over two different minimal ones. For instance, let $\Delta= K_{6,2}$ be the complete two-skeleton of the simplex on $6$ vertices, and let $\Gamma$ be a triangulation of the projective plane on $6$ vertices.  Then $\Delta$ is shellable and is also shelled over $\Gamma$. That $\Gamma$ is minimal CM follows from Corollary \ref{twofacet}.
\end{remark}

Next, we aim to prove that minimal CM complexes are acyclic. This is accomplished by showing a more general result. 

\begin{theorem}\label{facetdeath}
Suppose $\tilde{H}_{d-1}(\Delta) \ne 0$. Then there is a maximal facet $F_i$ of $\Delta$ so that the following hold: 

\begin{align}
\tag{1}
\dim \tilde{H}_{i-1}(\Delta_{F_i})&=\begin{cases} \dim \tilde{H}_{i-1}(\Delta) & \mbox{if } 0 \le i<d \\ \dim \tilde{H}_{i-1}(\Delta)-1 & \mbox{if } i=d \end{cases} \\ \tag{2}
f_{k-1}(\Delta_{F_i})&=\begin{cases} f_{k-1}(\Delta) & \mbox{if } 0 \le i<d \\ f_{k-1}(\Delta)-1 & \mbox{if } i=d \end{cases}  \\ \tag{3}
\depth \Delta&=\depth \Delta_{F_i}. 
\end{align}

\end{theorem}

\begin{proof}

Let $C_{\bullet}:0 \rightarrow C_{d-1} \xrightarrow{\partial_{d-1}}  \cdots \xrightarrow{\partial_2} C_1 \xrightarrow{\partial_1} C_0 \rightarrow 0$ be the associated chain complex of $\Delta$.  Choose a nonzero $\mathfrak{C} \in \tilde{H}_{d-1}(\Delta)=\ker \partial_{d-1}$.  Write $\mathfrak{C}= \displaystyle \sum_{i=1}^e r_iF_i$ where, without loss of generality, $r_1 \ne 0$.  Then $\partial_{d-1}(\mathfrak{C})=\displaystyle \sum_{i=1}^e r_i\partial_{d-1}(F_i)=0$ so $\partial(F_1)=\displaystyle \sum^e_{i=2}(-\dfrac{r_i}{r_1})\partial_{d-1}(F_i)$. Thus, if we remove $F_1$, its boundary remains, and so it follows that

\begin{equation}\label{skels}
\Delta^{(d-2)}=\Delta_{F_1}^{(d-2)}.
\end{equation}

So (1) and (2) are now immediate, and it only remains to show (3).  Suppose $\sigma \in \Delta_{F_1}$.  
Following from (\ref{skels}), we have $(\lk_{\Delta}(\sigma))^{(\dim \lk_{\Delta}(\sigma)-1)}=(\lk_{\Delta_{F_i}}(\sigma))^{(\dim \lk_{\Delta_{F_i}}(\sigma)-1)}$ and so $\lk_{\Delta_{F_1}}(\sigma)$ and $\lk_{\Delta}(\sigma)$ have the same homologies except potentially in top degree.  It follows that $\depth \Delta=\depth \Delta_{F_1}$.

\end{proof}

\begin{cor}\label{cor2}

Minimal CM complexes are acyclic.

\end{cor}

\begin{proof}

This is an immediate consequence of Lemmas \ref{isshelled} and Theorem \ref{facetdeath}.
\end{proof}

While highly acyclic CM complexes need not be minimal CM (a stack of simplices is $d-2$-acyclic but are shellable, so can never be minimal CM), one can provide some additional conditions under which they are.

\begin{theorem}\label{thm3}
If $\Delta$ is an \(\ell\)-fold acyclic CM simplicial complex and \(F\) is a facet of $\Delta$ that contains no more than \(\ell-1\) boundary ridges, then $\Delta_F$ is not CM. In particular, if every facet of \(\Delta\) contains no more than \(\ell-1\) boundary ridges, then \(\Delta\) is minimal CM.
\end{theorem}

\begin{proof}
Let \(F\) be a facet of \(\Delta\), and let \(\{R_1,\ldots,R_j\}\) be the set of boundary ridges contained in \(F\). We define \(v_i = F \setminus R_i\), and \(\sigma = v_1\ldots v_j\). Since \(\Delta\) is \(\ell\)-fold acyclic with \(|\sigma|=j \leq \ell-1 < \ell\), \(\lk_\Delta(\sigma)\) is acyclic. Furthermore, \(F\setminus \sigma\) has no boundary ridges in \(\lk_\Delta(\sigma)\).

We first note that
\[h_{d-j}(\lk_\Delta(\sigma))=\sum^{d-j}_{i=0} (-1)^{d-j-i}f_{i-1}(\lk_\Delta(\sigma))=(-1)^{d-j-1}\tilde{\chi}(\lk_\Delta(\sigma))=0.\]

Since every ridge in \(F\setminus \sigma\) in \(\lk_\Delta(\sigma)\) is contained in some other facet of \(\lk_\Delta(\sigma)\),

\[f_{i-1}(\lk_{\Delta_F}(\sigma))=\begin{cases} f_{i-1}(\lk_\Delta(\sigma)) & \mbox{if } 0 \le i<d-j \\ f_{i-1}(\lk_\Delta(\sigma))-1 & \mbox{if } i=d-j. \end{cases}\]

This implies 
\begin{align*}
h_{d-j}(\lk_{\Delta_F}(\sigma))&=\sum^{d-j}_{i=0} (-1)^{d-j-i}f_{i-1}(\lk_{\Delta_F}(\sigma))\\
&=\sum^{d-j-1}_{i=0} (-1)^{d-j-i}f_{i-1}(\lk_{\Delta_F}(\sigma))+f_{d-j}(\lk_{\Delta_F}(\sigma))\\
&=\sum^{d-j-1}_{i=0} (-1)^{d-j-i}f_{i-1}(\lk_\Delta(\sigma))+f_{d-j}(\lk_{\Delta_F}(\sigma))\\
&=\sum^{d-j-1}_{i=0} (-1)^{d-j-i}f_{i-1}(\lk_\Delta(\sigma))+f_{d-j}(\lk_\Delta(\sigma))-1\\
&=h_{d-j}(\lk_\Delta(\sigma))-1\\
&=-1.\\
\end{align*}

Thus $\lk_{\Delta_F}(\sigma)$ has a negative entry in its $h$-vector, and so cannot be CM.  In particular, $\Delta_F$ is not CM, so $\Delta$ is minimal CM.

\end{proof}

Setting \(\ell =1\) immediately gives the following special case.

\begin{cor}\label{twofacet}
If \(\Delta\) is an acyclic CM pseudomanifold, then \(\Delta\) is minimal CM.
\end{cor}

We now consider the relationship between the minimal CM property and the strongly non-shellable property for balls. Strongly non-shellablility has been used quite frequently in the study of non-shellable balls (see e.g. \cite{DK78,Ha00,Lu04-2,Lu04,Zi98}). It is defined as follows:

\begin{defn}
We say a ball \(B\) is strongly non-shellable if \(B_F\) is a non-ball for any facet $F \in B$.

\end{defn}

\begin{remark}

A strongly non-shellable ball is often defined (as in \cite{Zi98}) as a ball $B$ that does not contain a free facet, i.e., a facet $F$ such that $\langle F \rangle \cap \partial B$ is a ball of dimension $d-2$.  It's easy to see this definition is equivalent to the one we provide; it follows immediately from \cite[Proposition 2.4 (iii)]{Zi98} that any ball with a free facet cannot be strongly non-shellable in our notion.  On the other hand, if \(B_F\) is a ball, then, as in the proof of Proposition \ref{strongnonshell}, $\langle F \rangle \cap \partial B$ is generated by the ridges not containing $\sigma$, so it must also be a ball.

\end{remark}

\begin{prop}\label{strongnonshell}
For a ball \(B\), the following are equivalent:
\begin{enumerate}
\item[$(1)$] \(B\) is minimal CM.
\item[$(2)$] \(B\) is strongly non-shellable.
\end{enumerate}
\end{prop}

\begin{proof}
We first show $(1)$ implies $(2)$. Suppose \(B\) is minimal CM. Then removing any facet of \(B\) gives a complex that is not CM. This certainly can't be a ball, so $B$ is strongly non-shellable.

We now show $(2)$ implies $(1)$.  Suppose \(B\) is strongly non-shellable. Let \(F\) be a facet of \(B\). By Lemma \ref{isshelled}, since \(B\) is CM, \(\langle F \rangle \cap B_F\) is pure of dimension \(d-2\). Then the facets of \(\langle F \rangle \cap B_F\) are exactly the ridges of $\langle F \rangle$ which are non-boundary ridges in \(B\); label these $R_1,\dots,R_m$ and let \(\sigma=\bigcap_{i=1}^m R_i\).  

Each ridge of $\langle F \rangle$ that is not in $B_F$ must then be in $\langle F \rangle \cap \partial B$, which cannot be a ball by \cite[Proposition 2.4 (iii)]{Zi98}.  Thus \(\langle F \rangle \cap \partial B\) must have a facet $\tau$ that is not a ridge of $\langle F \rangle$.  It follows that $\tau \in \langle F \rangle \cap B_F$.  In particular, we have $\sigma \in \langle F \rangle \cap \partial B$. 

Now, note that 
\[\lk_{\langle F \rangle \cap B_F}(\sigma)=\langle R_1 \setminus \sigma,\dots,R_m \setminus \sigma \rangle=\partial \langle F \setminus \sigma \rangle= \partial \lk_{\langle F \rangle} (\sigma).\]
This shows that $F \setminus \sigma$ is a facet of $\lk_B(\sigma)$ which contains no boundary ridges of $\lk_B(\sigma)$.  Since $\sigma \in \partial B$, $\lk_{B}(\sigma)$ is acyclic, so we may apply Theorem \ref{thm3} to conclude that $\lk_B(\sigma)_{F \setminus \sigma}=\lk_{B_F}(\sigma)$ is not CM.  Thus \(B_F\) is not CM, and we are done.

\end{proof}

\section{Building New Minimal CM Complexes From Old Ones}\label{newmCM}

In this section we provide several results which show operations such as gluing or taking joins can be used to construct new examples of minimal CM complexes. 

We begin with results on gluing, starting with the following corollary of Theorem \ref{thm3}.

\begin{cor}\label{glue1}

Suppose $\Delta_1$ and $\Delta_2$ are CM acyclic complexes.  Set $\Gamma=\Delta_1 \cap \Delta_2$.  Suppose $\partial D_1,\partial D_2 \subseteq \Gamma$, that $\dim(\partial \Delta_1)=\Gamma=\dim(\partial \Delta_2)$, and that $\Gamma$ is acyclic and CM.  Then $\Delta_1 \cup \Delta_2$ is minimal CM.

\end{cor}

\begin{proof}
The assumptions and construction ensures that any ridge of $\Delta = \Delta_1\cup \Delta_2$ is contained in at least two facets. Thus we can appeal to \ref{twofacet}.

\end{proof}





For our next gluing result, we need the following dual notion of minimal CM.

\begin{defn}

We say $\Delta$ is \textbf{strongly CM} if $\Delta$ is CM and $\Delta_{F_i}$ is CM for any $i$.

\end{defn}

\begin{prop}\label{glue2}

If $\Delta$ and $\Gamma$ are minimal CM of dimension $d-1$ and $\Delta \cap \Gamma$ is strongly CM of dimension $d-1$, then $\Delta \cup \Gamma$ is minimal CM. 
\end{prop}

\begin{proof}

We have an exact sequence $0 \to k[\Delta \cup \Gamma] \to k[\Delta] \oplus k[\Gamma] \to k[\Delta \cap \Gamma] \to 0$ and then $k[\Delta \cup \Gamma]$ is CM by the depth lemma. Now let $F$ be a facet of $\Delta \cap \Gamma$.  Without loss of generality, assume $F \in \Delta$.  If $F \notin \Delta \cap \Gamma$, then we have the exact sequence 

\[0 \to k[\Delta_F \cup \Gamma] \to k[\Delta_F] \oplus k[\Gamma] \to k[\Delta \cap \Gamma] \to 0.\] 
Since $k[\Delta_F]$ is not CM, neither is $k[\Delta_F \cup \Gamma]$. Otherwise, $F \in \Delta \cap \Gamma$, and we have the exact sequence 
\[0 \to k[(\Delta \cup \Gamma)_F] \to k[\Delta_F] \oplus k[\Gamma_F] \to k[(\Delta \cap \Gamma)_F] \to 0.\]

Since $\Delta \cap \Gamma$ is CM and $k[\Delta_F] \oplus k[\Gamma_F]$ is not, $k[(\Delta \cup \Gamma)_F]$ is not CM, completing the proof.

\end{proof}

\begin{remark}

The proof of Proposition \ref{glue2} does not actually require $\Delta \cap \Gamma$ to be strongly CM, only that $\Delta \cap \Gamma$ have dimension and depth $d-1$, and that $(\Delta \cap \Gamma)_F$ be CM for every facet $F \in \Delta \cap \Gamma$.

\end{remark}

We end this section by showing that the join of a minimal CM complex and another (not necessarily minimal) CM complex is minimal CM. 

\begin{theorem}\label{thm4}
If \(\Delta\) is minimal CM and \(\Gamma\) is CM, then \(\Delta \star \Gamma\) is minimal CM.
\end{theorem}

\begin{proof}
First we note that \(\Delta \star \Gamma\) is CM. Let \(F\) be a facet of \(\Delta \star \Gamma\). We may write \(F = F' \star G\) for some facets \(F'\) of \(\Delta\) and \(G\) of \(\Gamma\). Since \(\Delta_{F'}\) is not CM, there exists \(\sigma \in \Delta_{F'}\) such that \(\tilde{H}_i(\lk_{\Delta_{F'}}(\sigma)) \ne 0\) for some \(i<\dim(\Delta)-|\sigma|\).

We now show that \(L=\lk{\Delta_{F'}}(\sigma)\) and \(L'=\lk_{(\Delta\star\Gamma)_{F}}(\sigma \star G))\) are isomorphic posets. If \(\tau \in L\), it is immediate that \(\tau \star \emptyset \in L'\). Furthermore, if \(\tau \star \emptyset \in L'\), then \(\tau \in L\). Suppose \(\tau \star \gamma \in L'\), then \(\gamma \cup G \in \Gamma\) and \(\gamma \cap G = \emptyset\). This implies that \(\gamma = \emptyset\). So \(L\) and \(L'\) are isomorphic posets.

With this isomorphism, we see that \(\tilde{H}_i(\lk_{(\Delta\star\Gamma)_{F}}(\sigma \star G)) \ne 0\) for some \(i<\dim(\Delta \star \Gamma)-|\sigma \star G|\). Thus \((\Delta \star \Gamma)_F\) is not CM by Reisner's Criterion (\cite[Theorem 1]{Re76}), and therefore, \(\Delta \star \Gamma\) is minimal CM.
\end{proof}

\section{Examples}\label{examples}

In this section,  we consider some notable examples of CM complexes from the literature, and show that they are minimal CM. We expect that this is far from a complete list of minimal CM examples currently published. 

A large class of minimal CM complexes are those that satisfy the conditions of Corollary \ref{twofacet}. The following complexes fall in this class:

\begin{itemize}
\item Triangulations of \(\mathbb{RP}^{2n}\) (over \(k\) of characteristic not 2)
\item The dunce hat
\item Bing's House with 2 rooms \cite{Ha99}
\item The pastry \cite{Do18}
\end{itemize}


The following are all strongly non-shellable balls, which are minimal CM by Proposition \ref{strongnonshell}.

\begin{itemize}
\item Rudin's Ball \cite{Ru58}
\item \(B^3_{16,48}\), \(B^3_{12,37,a}\), and  \(B^3_{12,37,b}\) \cite{Lu04}
\item \(B_{3,9,18}\) \cite{Lu04-2}
\item Ziegler's Ball \cite{Zi98}
\end{itemize}

This next class of minimal CM complexes are constructible complexes which are not themselves balls. Each of these were verified to be minimal CM by applying Theorem \ref{thm3}. 

\begin{itemize}
\item The complex \(C_3\) in \cite{DG16}, a non-partitionable CM complex
\item The complex \(C_3\) in \cite{JV17}, a balanced non-partitionable CM complex
\item The complex \(\Omega_3\) in \cite{DG18}, a 2-fold acyclic complex with no decomposition into rank 2 boolean intervals
\end{itemize}

Each of these complexes is a counterexample to an associated conjecture in the literature; see the references for more details. They are each the result of gluing many copies of a CM complex along a CM subcomplex, a similar process to  Proposition \ref{glue2}.

\section*{Acknowledgments}
We thank Bruno Benedetti for insightful conversations regarding constructions of interesting balls and nonshellable complexes. The first author also acknowledges partial support from Simons Collaboration Grant FND0077558.  We thank the Mathematics Department at the University of Kansas for the excellent working and collaborating conditions.

\bibliographystyle{plain}
\bibliography{mybib}

\end{document}